\documentclass[a4paper,10pt,leqno,english]{smfart}
\usepackage{aeguill}
\usepackage{enumerate}
\usepackage{amssymb,amsmath,latexsym,amsthm}
\usepackage[T1]{fontenc}
\usepackage{smfthm}      
\usepackage{geometry}   
\usepackage{url}      
\usepackage[frenchb, english]{babel}
\usepackage[utf8]{inputenc}   
\usepackage{mathrsfs}
\usepackage{relsize}
\usepackage{xcolor}
\usepackage{comment}
\definecolor{violet}{rgb}{0.0,0.2,0.7}
\definecolor{rouge2}{rgb}{0.8,0.0,0.2}
\usepackage{hyperref}

\hypersetup{
    bookmarks=true,         
    unicode=false,          
    pdftoolbar=true,        
    pdfmenubar=true,        
    pdffitwindow=false,     
    pdfstartview={FitH},    
    pdftitle={},    
    pdfauthor={},     
    colorlinks=true,       
   linkcolor=violet,          
    citecolor=rouge2,        
    filecolor=black,      
    urlcolor=cyan}           
\newcommand{\R}{\mathbb{R}}
\newcommand{\CC}{\mathbb{C}}
\newcommand{\Q}{\mathbb{Q}}
\newcommand{\Z}{\mathbb{Z}}

\newcommand{\B}{\mathbb{B}}

\newcommand{\vp}{\varphi}

\renewcommand{\O}{\mathcal{O}}
\newcommand{\Ox}{\mathcal{O}_{X}}
\newcommand{\ep}{\varepsilon}

\newcommand{\la}{\langle}
\newcommand{\ra}{\rangle}

\renewcommand{\ge}{\geqslant}
\renewcommand{\le}{\leqslant}
\newcommand{\Ric}{\mathrm{Ric} \,}
\newcommand{\om}{\omega}
\newcommand{\ddc}{dd^c}

\newcommand{\vpt}{\varphi_t}

\newcommand{\xreg}{X_{\rm reg}}
\newcommand{\vpe}{\varphi_{\varepsilon}}

\newcommand{\te}{\tau_{\varepsilon}}

\newcommand{\sd}{\mathrm{Supp}(D)}

\newcommand{\D}{D}
\newcommand{\ome}{\om_{\varepsilon}}
\newcommand{\omt}{\om_{t}}
\newcommand{\omr}{\om_{\rho}}
\newcommand{\omre}{\om_{\rho_{\ep}}}
\newcommand{\Supp}{\mathrm {Supp}}

\newcommand{\sm}{\setminus}
\newcommand{\tr}{\mathrm{tr}}

\newcommand{\psie}{\psi_{\ep}}
\newcommand{\rhoe}{\rho_{\ep}}
\newcommand{\pp}{\psi^{+}}
\newcommand{\psm}{\psi^{-}}
\newcommand{\ppm}{\psi^{\pm}}
\newcommand{\amp}{\mathrm{Amp}(\alpha)}
\newcommand{\bp}{\B_+(\alpha)}
\newcommand{\vol}{\mathrm{vol}}
\newcommand{\MA}{\mathrm{MA}}
\newcommand{\MAl}{(\mathrm{MA}_{\lambda})}
\newcommand{\KEl}{(\mathrm{KE_{\lambda}})}

\newtheorem*{thma}{Theorem A}
\newtheorem*{thmb}{Theorem B}

\renewcommand{\theequation}{\thesection  .\!\! \arabic{equation}}

\begin{document}
\title[Conic Kähler-Einstein metrics on klt pairs]{Kähler-Einstein metrics with cone singularities on klt pairs}

\date{\today}
\author{Henri Guenancia}
\address{Institut de Mathématiques de Jussieu \\
Université Pierre et Marie Curie \\
 Paris 
\& Département de Mathématiques et Applications \\
\'Ecole Normale Supérieure \\
Paris}
\email{guenancia@math.jussieu.fr}
\urladdr{www.math.jussieu.fr/~guenancia}
\maketitle
\tableofcontents

\section{Introduction}

Let $(X,D)$ be a Kawamata log-terminal pair (shortened in \textit{klt}), i.e. $X$ is a normal projective variety over $\CC$ of dimension $n$, and $D$ is an arbitrary $\Q$-divisor such that $K_X+D$ is $\Q$-Cartier, and for some (or equivalently any) log-resolution $\pi : X'\to X$, we have: 
\[K_{X'}=\pi^*(K_X+D)+\sum a_i E_i\]
where $E_i$ are either exceptional divisors or components of the strict transform of $D$, and the coefficients $a_i$ satisfy the inequality $a_i> -1$. \\

For such a pair $(X, D)$, there exists a natural notion of Kähler-Einstein metric developed by Eyssidieux, Guedj and Zeriahi in \cite{EGZ} for the non-positively curved case (ie $K_X+D$ ample or trivial), and extended by Berman, Boucksom and the previous authors in \cite{BBEGZ} for log-Fano varieties ($-(K_X+D)$ ample). Moreover, if $K_X+D$ is merely big, there exists also a unique Kähler-Einstein metric thanks to the finite generation of the ring $\oplus_{m\ge 0} H^0(X,m(K_X+D))$ proved by \cite{BCHM}, cf. e.g. \cite[section 6]{BEGZ} or the explanations below Theorem A. In that case, there is a particular Zariski open subset of $X$ introduced by Boucksom in \cite{DZD} called ample locus and denoted by $\mathrm{Amp}(K_X+D)$ which plays a special role : it is an open subset where the cohomology class $c_1(K_X+D)$ "looks like" a Kähler class. We refer to section \ref{ssct} for the definitions of those objects. \\

Those objects are currents living on a log-resolution of the pairs, and are very particular from the point of view of pluripotential theory (they have finite energy). Moreover, those currents are shown to induce genuine Kähler-Einstein metrics on the Zariski open subset $X_0:=\xreg\setminus \sd$. However, little is known on the behavior of the Kähler-Einstein metrics near the degeneracy locus $X\setminus X_0$. \\ 

As understanding the degeneracy of the Kähler-Einstein metric near the singularities of $X$ seems out reach for the moment, we will focus on the behavior near the support of $D$ intersected with the regular locus of $X$. It turns out that this is equivalent to control the singularities of the solution of a degenerate Monge-Ampère equation on a \textit{smooth manifold}, the term "degenerate" meaning here that the solution lives in a class which is not Kähler (but at least big, or nef and big), and that the right-hand side is not smooth (cf Theorem B). To avoid the singularities of the pair $(X,D)$, we will restrict to study the Kähler-Einstein metric on the open subset of $X$ which is not affected by the log-resolution: this is the set of points $x$ where $(X,D)$ is log-smooth at $x$, in the sense that $x$ is a smooth point of $X$, and there is a (small) analytic neighborhood $U$ of $x$ such that $\Supp(D) \cap U= \{z_1 \cdots z_k=0\}$ for some holomorphic coordinates $z_1, \ldots, z_n$. \\

\begin{thma}
Let $(X,D)$ be a klt pair, and let $\mathrm{LS}(X,D):= \{x\in X; (X,D)$ is log-smooth at $x\}$. We assume that the coefficients of $D$ are in $[1/2, 1)$.
\begin{enumerate}
\item[$(i)$] If $K_X+D$ is big, then the Kähler-Einstein metric of $(X,D)$ has cone singularities along $D$ on $\mathrm{LS}(X,D)\cap \mathrm{Amp}(K_X+D)$. 
\item[$(ii)$] If $-(K_X+D)$ is ample, then any Kähler-Einstein metric of $(X,D)$ has cone singularities along $D$ on $\mathrm{LS}(X,D)$.
\end{enumerate}
\end{thma}

In order to relate this result to previous works, we should mention that a result like Theorem A was already known when $(X,D)$ is a log-smooth pair with $K_X+D$ ample. More precisely, this result was obtained in \cite{Brendle,JMR}  whenever $D$ is irreducible and in \cite{CGP} when the coefficients of $D$ are in $[1/2, 1)$. Moreover, the general (Kähler) case has been announced by R. Mazzeo and Y. Rubinstein in \cite{MR}. \\

Let us mention that we also prove a similar result for the Ricci-flat case : if under the same set-up, we are given a nef and big class $\alpha$, and assume that $c_1(K_X+D)$ is trivial, then the Ricci-flat metric in $\alpha$ has cone singularities along $D$ on  $\mathrm{LS}(X,D)\cap \mathrm{Amp}(K_X+D)$. As the result is a bit less natural, we chose not to include it in the previous Theorem.

We should also mention that $\mathrm{LS}(X,D)$ is a Zariski open subset of $X$ with complement of codimension at least $2$, as it contains the intersection of $X_{\rm reg}$ and the regular locus of $D_{\rm red}$.\\

Let us explain case $(i)$ in Theorem A. For the moment, we do not know any regularity result for solutions of Monge-Ampère equations in general big classes. But in the case of adjoint bundles coming from klt pairs, the main theorems of the Minimal Model Program (MMP) theory enable us to reduce to the semipositive and big case which is far better understood. Let us now be more precise:\\

 Let $(X,D)$ be a klt pair of log-general type, ie $K_X+D$ is assumed to be big. By the fundamental results of \cite[Theorem 1.2]{BCHM}, we know that $(X,D)$ has a log-canonical model $f:X \dashrightarrow X_{\rm can}$ where $X_{\rm can}$ is a projective normal variety and $f$ is a birational contraction such that $f_*(K_X+D)=K_{X_{\rm can}}+D_{\rm can}$ is ample (here $D_{\rm can}:=f_*D$). If $\mu: Y \to X$, $\nu: Y \to X_{\rm can}$ is a resolution of the graph of $f$, then \[\mu^* (K_X+D)=f^*(K_{X_{\rm can}}+D_{\rm can})+E\] for some effective $\nu$-exceptional divisor $E$.   
A consequence of this Zariski decomposition is that every positive current $T$ in $c_1(K_X+D)$ comes from a unique positive current $S\in c_1(K_{X_{\rm can}}+D_{\rm can})$ in the following way: $\mu^*T= \nu^*S + [E]$. In particular the Kähler-Einstein metric on $(X,D)$ constructed in \cite[Theorem 6.4]{BEGZ} (see also Theorem \ref{thm:l1} for the general case of a klt pair) corresponds in this way to the Kähler-Einstein metric on $(X_{\rm can}, D_{\rm can})$ constructed in \cite[Theorem 7.12]{EGZ}. In particular, this last metric is smooth outside of the singular locus of $X_{\rm can}$ and $\mathrm{Supp}(D_{\rm can})$. \\

Moreover, as the log-canonical model $f:X \dashrightarrow X_{\rm can}$ induces an isomorphism from the ample locus of $K_X+D$ onto its image (this is a general property for the maps attached to big linear systems, which follows directly from the definition of $\mathbb B_+(K_X+D)$), it is enough for our matter to understand the Kähler-Einstein metric of $(X_{\rm can}, D_{\rm can})$. So we are reduced to working in semipositive and big classes as explained above. \\

Once this reduction is done, we can express the problem in terms of Monge-Ampère equations (cf section \ref{red}); the framework is the following one: let $X$ be a compact Kähler manifold of dimension $n$, $\alpha$ a nef and big class class and $\theta\in \alpha$ a smooth representative. Let $E=\sum c_j E_j$ be an effective $\R$-divisor with snc support such that $\alpha -E$ is Kähler. Let also $D= \sum a_i D_i$ an effective divisor with snc support such that $E$ and $D$ have no common components, and that $E+D$ has snc support. 

We write $X_0=X\setminus (\Supp(E) \cup \Supp(D))$, we choose non-zero global sections $t_j$ of $\O_X(E_j)$ and $s_i$ of $\O_X(D_i)$, and we choose some real numbers $b_j>-1$. We choose some smooth hermitian metrics on those bundles which we normalize so that $\int_X \prod |t_j|^{2b_j} \prod |s_i|^{-2a_i}dV=\vol(\alpha)$. Finally, if $\vp$ is a $\theta$-psh function, we denote by $\MA(\vp)$ the non-pluripolar product $\la (\theta+\ddc \vp)^n\ra$ in the sense of \cite{BEGZ}, cf. section \ref{bcc2}.\\

\begin{thmb}
\phantomsection
\label{thm:cone}
We assume that the coefficients of $D$ satisfy the inequalities $a_i \ge 1/2$ for all $i$. 
Then any solution with full Monge-Ampère mass of 
\[\MA(\vp)= \prod |t_j|^{2b_j} \frac{e^{\lambda \vp}dV}{\prod |s_i|^{2a_i}}\]
defines a smooth metric on $X_0$ having cone singularities along $D$.
\end{thmb}

Let us mention that if $\lambda \ge0$, such a solution always exists and is unique. If $\lambda<0$, there might be no solution, or on the contrary, many ones. 

In order to deduce Theorem A from Theorem B, one considers a log-resolution of the pair $(X,D)$ which also computes the augmented base locus as an SNC divisor meeting the strict transform of $D$ properly. The Monge-Ampère equation giving the Kähler-Einstein metrics pulls back to this resolution and has exactly the form considered above. Finally, we observe that a log-resolution of a pair $(X,D)$ is an isomorphism on the previously defined log-smooth locus $\mathrm{LS}(X,D)$, so we are done.\\

One may remark that to deduce Theorem A from Theorem B, it would have been enough to assume $\alpha$ semi-positive and big. As the proof of the nef and big case is not really more complicated, we chose to state the Theorem in this slightly greater generality.
We should also add that these last results of are expected to be valid in the more general case where $\alpha$ is only big and not necessarily nef. However, even when $D=0$, we were not able to prove that the Kähler-Einstein metrics are smooth on the ample locus of $\alpha$, and new ideas shall probably be needed to settle this question. \\

In fact, using the same techniques appearing in the proof of Theorem B, we can formulate a slightly more general result (cf. Theorem \ref{gen}): from the Monge-Ampère point of view (Theorem B), we do not need that the factors $|t_j|^{-2b_j}$ come from the augmented base locus $E$. Therefore, even if $D$ has coefficients in $(0,1)$ and not only $[1/2,1)$, one can still prove that the Kähler-Einstein metrics in Theorem A will have cone singularities along $\sum_{a_i\ge 1/2}D_i$ when restricted to $X_{\rm reg} \cap \mathrm{LS}(X,D) \setminus \Supp(\sum_{a_i<1/2} D_i)$ (and intersected with $\mathrm{Amp}(K_X+D)$ in case $(i)$).

\section*{Acknowledgements}

I am indebted to Sébastien Boucksom who has patiently and thoroughly read the preliminaries versions of this work. His enlightening suggestions and remarks helped a lot to improve the exposition and the content of the present article.

\section{Monge-Ampère equations in big cohomology classes}
\label{bcc}

In this section, we recall some generalities on big cohomology classes on a compact Kähler manifold, and then give an outline of the paper by Boucksom, Eyssidieux, Guedj and Zeriahi \cite{BEGZ} which we are going to rely on.

\subsection{Generalities on big cohomology classes}

We start with a compact Kähler manifold $X$ of dimension $n$, and we consider a class $\alpha \in H^{1,1}(X,\R)$ which is big. By definition, this means that $\alpha$ lies in the interior of the pseudo-effective cone, so that there exists a Kähler current $T\in \alpha$, that is a current which dominates some smooth positive form $\omega$ on $X$.\\

\subsubsection{The ample locus of $\alpha$}
\label{ssct}

One may define, following S. Boucksom \cite[\textsection 3.5]{DZD}, the \textit{ample locus} of $\alpha$, denoted $\amp$, which is the largest Zariski open subset $U$ of $X$ such that for all $x\in U$, there exists a Kähler current $T_x\in \alpha$ with analytic singularities such that $T_x$ is smooth in an (analytic) neighbourhood of $x$. A generalization of Kodaira's lemma asserts that for some modification $\pi:X'\to X$, one can write $\pi^*\alpha = \beta + E$ where $\beta$ is a Kähler class and $E$ is an effective divisor on $X'$. 
Then one has the following characterization of the complement of $\amp$, denoted $\B_+(\alpha)$ and called \emph{augmented base locus} in analogy with the case where $\alpha$ is a big class in the real Néron-Severi group $NS(X)\otimes \R$ (cf \cite{ELMNP}, \cite[Lemma 1.4]{BBP}):
\[\bp = \bigcap_{\pi^*\alpha-E \, \,  \mathrm{K\ddot{a}hler}} \pi(\Supp(E)) \]
where $E$ ranges over all effective $\R$-divisor in a birational model $\pi:X'\to X$ such that $\pi^*\alpha-E$ is a Kähler class.\\

At this point, two remarks need to be made. The first one is that the augmented base locus (or equivalently the ample locus) of a big class $\alpha$ can be computed by a single modification $\pi:X' \to X$. 

\noindent
Indeed, by the noetherianity of $X$ for the (holomorphic) Zariski topology, there exists a Kähler current $T\in \alpha$ with analytic singularities such that the singular locus of $T$ is exactly $\bp$. Resolving the singularities of $T$ (cf \cite[\textsection 2.6]{DZD}), one obtains a morphism $\pi:X' \to X$ such that $\pi^*T=\theta + [E]$ where $\theta \ge \pi^* \omega$ for any Kähler form $\omega$ on $X$ dominated by $T$, and $E=\sum a_i E_i$ is an effective $\R$-divisor lying above the singular locus of $T$ (so it is not necessarily exceptional because $\bp$ might have one-codimensional components). Moreover, as $\pi$ is a birational morphism between smooth varieties (actually we use here that $X$ is locally $\Q$-factorial) there exists some positive linear combination of exceptional divisors $F= \sum b_i E_i$ such that $-F$ is $\pi$-ample (cf \cite[II, ex. 7.11 (c)]{Har} or \cite[1.42]{Deb}). Therefore, for a sufficiently small $\ep$, the cohomology class of $\theta-\ep F$ contains a Kähler form, so that we have the following decomposition: 
\[\pi^* \alpha = \{\theta-\ep F\}+(E+\ep F)\]
with $\{\theta-\ep F\}$ Kähler, and $E+\ep F$ effective, with support equal to $\Supp(E)$. Therefore $\bp = \pi(\Supp(E))$, which shows that $\bp$ can be indeed computed by a single modification of $X$.\\

The second remark we would like to do about the notion of ample locus concerns the case when $\alpha=c_1(L)$ is the Chern class of a line bundle. In that case, there is no need to perform modifications of $X$ to compute $\bp$, as is shown in \cite[Remark 1.3]{ELMNP}: 
\[\mathbb B _+(L) = \bigcap_{\substack{L=A+E \\ A \,  \mathsmaller{a\!m\!pl\!e,} \,  E \ge 0}} \Supp(E) \]

\subsubsection{Currents with minimal singularities}

We will be very brief about this well-known notion, and refer e.g. to \cite[\textsection 2.8]{DZD}, \cite[\textsection 1]{BBGZ}, \cite{Ber2} or \cite{BD} for more details and recent results. 

\noindent
By definition, if $T,T'$ are two positive closed currents in the same cohomology class $\alpha$, we say that $T$ is less singular than $T'$ if the local potentials $\vp, \vp'$ of these currents satisfy $\vp' \le \vp + O(1)$. It is clear that his definition does not depend on the choice of the local potentials, so that the definition is coherent. In each (pseudo-effective) cohomology class $\alpha$, one can find a positive closed current $T_{\rm min}$ which will be less singular than all the other ones; this current is not unique in general; only its class of singularities is. Such a current will be called current with minimal singularities.  

One way to find such a current is to pick $\theta \in \alpha$ a smooth representative, and define then, following Demailly, the upper envelope 
\[V_{\theta}:= \sup \{\vp \,\,  \theta \mathrm{-psh} , \,  \vp \le 0 \,  \, \textrm{on} \, \, X\}\] 
Once observed that $V_{\theta}$ is $\theta$-psh (in particular upper semi-continuous), it becomes clear that $\theta+ \ddc V_{\theta}$ has minimal singularities. \\

\subsection{Non-pluripolar product and Monge-Ampère equations}
\label{bcc2}

\subsubsection{Non-pluripolar product}
In the paper \cite{BEGZ}, the \textit{non-pluripolar product} $T \mapsto \la T^n \ra $ of any closed positive $(1,1)$-current $T\in \alpha$ is shown to be a well-defined measure on $X$ putting no mass on pluripolar sets. Given now a $\theta$-psh function $\vp$, one defines its non-pluripolar Monge-Ampère by 
$\MA(\vp) := \la (\theta + \ddc \vp)^n \ra$. Then one can check easily from the construction that the total mass of $\MA(\vp)$ is less than or equal to the volume $\vol(\alpha)$ of the class $\alpha$:
\[ \int_X \MA(\vp) \le \vol(\alpha)\]
A particular class of $\theta$-psh functions that appears naturally is the one for which the last inequality is an equality. We will say that such functions (or the associated currents) have \textit{full Monge-Ampère mass}. 	\\

Let us consider now the case of $\theta$-psh functions with minimal singularities. By definition, they are locally bounded on $\amp$, so that one can consider on this open set their Monge-Ampère $(\theta+ \ddc \vp)^n$ in the usual sense of Bedford-Taylor. Then one can see that the trivial extension of this measure to $X$ coincide with $\MA(\vp)$ and satisfies
\[\int_X \MA(\vp) = \vol(\alpha)\]   
In particular, currents with minimal singularities have full Monge-Ampère mass, the converse being false however. An observation that dates back to S. Boucksom \cite[Proposition 3.6]{DZD} show that whenever $\alpha$ is nef and big, then the positive currents in $\alpha$ having minimal singularities automatically have zero Lelong numbers, or equivalently, using Skoda's integrability theorem \cite{Sko}, their potentials $\vp$ satisfy $e^{-\vp}  \in L^{p} $ for all $p \ge 1$. 

Very recently, a similar statement has been obtained for semi-positive and big classes by Berman, Boucksom, Eyssidieux, Guedj and Zeriahi. The precise statement is the following one: 

\begin{theo}[{\cite[Theorem 1.1]{BBEGZ}}]
\phantomsection
\label{zln}
Let $X$ be a normal compact complex space endowed with a fixed Kähler form $\om_0$. Let $\vp$ be an $\om_0$-psh function with full Monge-Ampère mass, and $\pi:X'\to X$ be any resolution of singularities of $X$. Then $\vp':= \vp \circ \pi$ has zero Lelong numbers everywhere. Equivalently, $e^{-\vp'}\in L^p(X')$ for all $p\ge 1$.
\end{theo}

This result will be very helpful for the proof of Theorem B. If we did not have it, then in the case $\lambda<0$, we should have added the assumption that $\vp$ has minimal singularities.

\subsubsection{Monge-Ampère equations in big cohomology classes}
\label{sct}

One of the main results of the paper \cite{BEGZ}, is that for every non-pluripolar measure $\mu$, there exists a unique positive current $T_{\mu}\in \alpha$ with full Monge-Ampère mass satisfying the Monge-Ampère equation \[\la T^n \ra = \mu\] 
The strategy of the proof is to consider approximate Zariski decompositions $X_k \overset{\pi_k}{\longrightarrow} X$ (with $\pi_k^*\alpha = \beta_k+[E_k]$ where $\beta_k$ is Kähler and $E_k$ effective) and to solve $\la S_k^n \ra = \frac{\vol(\beta_k)}{\vol(\alpha)} \pi_k^* \mu$ with $S_k \in \beta_k$, which is possible thanks to the main result of \cite{GZ07} ($\beta_k$ is Kähler). Then one needs to prove that $T_k:= (\pi_k)_*(S_k+[E_k])$ converges to some current $T$ with full Monge-Ampère mass solution of the initial equation.\\

In that same paper, the authors of \cite{BEGZ} obtain $L^{\infty}$ estimates of the potential of the solution $T$ whenever the measure $\mu=f dV$ has $L^{1+\ep}$-density with respect to the Lebesgue measure. More precisely, they get the following result \cite[Theorem 4.1]{BEGZ}: the normalized potential $\vp$ (ie $\max_X \vp = 0$) solution of $\MA(\vp)=\mu$ satisfies \[\vp \ge V_{\theta} -M ||f||^{1/n}_{L^{1+\ep}}\] where $V_{\theta}$ is the upper envelope $V_{\theta}:=\sup\{\psi \, \, \theta \mathrm{-psh}, \psi \le 0 \, \, \mathrm{on} \, \, X\}$ defined in the previous section, and $M$ depends only on $\theta, dV$ and $\ep$.\\

Finally, if $\mu=dV$ is now a smooth volume form, and under the additional assumption that $\alpha$ is nef and big, then \cite[Theorem 5.1]{BEGZ} asserts that the solution $T_{\mu}$ is smooth on the ample locus $\amp$. Very little is known however about the behavior of $T_{\mu}$ along $\bp$, even in the case where $\alpha$ is semi-positive (and big).

\subsection{The equation \texorpdfstring{$\MA(\vp)=e^{\vp}\mu$}{MA(phi)=e^{phi} mu}}
In this section, we focus on the equation $\MA(\vp)=e^{\vp}\mu$, where $\mu$ is a non-pluripolar measure. As we explained in the introduction, this equation is related to negatively curved Kähler-Einstein metrics. \\ 

Whenever $\mu=dV$ is a smooth volume form, we have at our disposal \cite[Theorem 6.1]{BEGZ} which garantees that the previous equation admits a unique solution $\vp$ $\theta$-psh with full Monge-Ampère mass (we are still assuming that $\alpha=\{\theta\}$ is a big class).\\
In fact, their proof can be readily adapted to the case where $\mu$ has a $L^{1+\ep}$ density with respect to the Lebesgue measure:
\begin{theo}
\phantomsection
\label{thm:l1}
Let $X$ be a compact Kähler manifold of dimension $n$, $\alpha\in H^{1,1}(X,\R)$ a big class, and $\mu=f dV$ a volume form with density $f \ge 0$ belonging to $L^{1+\ep}(dV)$ for some $\ep >0$. Then there exists a unique $\theta$-psh function $\vp$ with full Monge-Ampère mass such that $\la (\theta+\ddc \vp)^n \ra =e^{\vp} \mu$. Furthermore, $\vp$ has minimal singularities.
\end{theo}

\begin{proof}[Sketch of Proof]
The proof is almost the same as the one of \cite[Theorem 6.1]{BEGZ}, so we only give the main ideas. We may assume without loss of generality that $\mu$ has total mass $1$, and consider $\mathcal C$ the subset of $L^1(X,\mu)$ consisting of all $\theta$-psh functions $\psi$ normalized by $\sup_X \psi = 0$. Indeed, $PSH(X,\theta) \subset L^1(X, \mu)$ because any $\theta$-psh function is in $L^{p}_{\mathrm{loc}}$ for every $p>0$, so applying Hölder's inequality with $p=1+\frac 1 {\ep}$, we obtain the result. This set is convex, compact, so there exists $C>0$ such that $\int_X \psi d\mu \ge -C$ for all $\psi \in \mathcal C$, as explained in \cite[Proposition 1.7]{GZ05}. By convexity, one deduces in particular that $\log \int_X e^{\psi}d\mu \ge -C$ for all $\psi \in \mathcal C$.\\

For all $\psi \in \mathcal C$, the measure $e^{\psi} \mu$ has uniform $L^{1+\ep}$ density with respect to $dV$ because $\psi \le 0$. Therefore, \cite[Theorems 3.1 \& 4.1]{BEGZ} ensure the existence of a unique function $\Phi(\psi) \in \mathcal C$ such that $\MA(\Phi(\psi))=e^{\psi+c_{\psi}} \mu$, where $c_{\psi}=\log \vol(\alpha)- \log \int_X e^{\psi}d\mu \le \log \vol(\alpha)+C$, and $\Phi(\psi) \ge V_{\theta}-M$ for some uniform $M$.\\

\cite[Lemma 6.2]{BEGZ} shows that the map $\Phi: \mathcal C \to \mathcal C$ is continuous, therefore it has a fixed point $\psi\in \mathcal C$ by Schauder's fixed point theorem, and one concludes setting $\vp:=\psi+c_{\psi}$. As any $\theta$-psh function is bounded above on $X$, every solution of our equation has minimal singularities, and thus uniqueness follows from the straightforward generalization of \cite[Proposition 6.3]{BEGZ} in the setting of measures with $L^{1+\ep}$ density with respect to the Lebesgue measure. 
\end{proof}

\begin{rema}
\phantomsection
\label{rema2}
Let us note that it follows from the proof of this theorem that there exists $M$ depending only on $\theta, dV$ and $\ep$ such that the solution $\vp$ of $\la (\theta+\ddc \vp)^n \ra =e^{\vp} \mu$ satisfies $M \ge \vp \ge V_{\theta}-M ||f||^{1/n}_{L^{1+\ep}}$. Indeed, the only point is to control the constant $C$ appearing in the previous proof, but $C$ is bounded by $\sup \{||\psi||_{L^{1+1/\ep}} \cdotp ||f||_{L^{1+\ep}} \, ; \, \psi \in \mathcal C \}$ which is finite by compactness of $\mathcal C$ and equivalence of the $L^1$ and $L^p$ topology for quasi-psh functions.
\end{rema}

\section{Cone singularities for Kähler-Einstein metrics}

\subsection{Singular Kähler-Einstein metrics}
\label{red}

As we mentioned in the introduction, one can define the notion of Kähler-Einstein metric attached to any klt pair $(X,D)$. We refer e.g. to \cite{EGZ} or \cite{BBEGZ} for the definition in the non-positively curved case and in the log-Fano case respectively. What is important to remember about those objects is that they are currents on the singular variety $X$ which satisfy on any log-resolution $(X',D')$ ($D'$ being given by the identity $K_{X'}+D'=\pi^*(K_X+D)$) a Monge-Ampère equation of the form 
\[\MA(\phi) = e^{\pm(\phi-\phi_{D'})}  \leqno{(\mathrm{KE})} \] 
in the negatively or positively curved case (and a similar equation in the Ricci-flat case). Here $\phi$ is a (singular) psh weight on $\pm c_1(K_{X'}+D')$ (or in some given semiample and big class), $\phi_{D'}$ is a (singular) weight on the $\R$-line bundle $\Ox(D')$ such that $\ddc \phi_{D' }= [D']$, and $\MA$ is the non-pluripolar Monge-Ampère operator. \\

A Kähler-Einstein metric $\om=\ddc \phi $ attached to $(X,D)$ shall satisfy the equation 
\[\Ric \om = \mp \om + [D]  \]
(whenever $\Ric \om = -\ddc \log \om^n$ makes sense) or equivalently $\Ric \om' = \mp \om' + [D'] $ where $\om'=\pi^* \om$ for a log-resolution $\pi:X'\to X$ of $(X,D)$. \\

It is not completely clear that any Kähler-Einstein metric (as previously defined) should be smooth on $X_{\rm reg} \setminus \Supp(D)$, or equivalently on $X' \setminus \Supp(D')$. This work has been done for the log-Fano case in \cite{BBEGZ} using an estimate appearing in \cite{Paun}. We will also follow this strategy to obtain the smoothness on the suitable locus (cf Remark \ref{smo}). This strategy will also enable us to establish the cone singularities of the Kähler-Einstein metric).\\

There are two difficulties arising when one wants to understand the solutions of the Monge-Ampère equation $\KEl$: first of all, the right-hand side is singular, and also the class in which one looks for a solution is no more ample (or Kähler) but merely semi-ample and big. However, working with those weak positivity notions has the advantage that our equation is invariant by modification, and can be read equally on any log-resolution. Therefore, one can assume that the augmented base locus of $\pm \pi^*(K_X+D)$ (or $\pi^*\alpha$ for $\alpha$ a Kähler class on $X$, in the Ricci-flat case) is given by a divisor $E$ (in a sense to make clear) having simple normal crossing support, and meeting the strict transform of $D$ also normally. We will see shortly that this context is well adapted for our purposes. \\

Let us recall that in the case where $(X,D)$ is a log-smooth pair (with $D$ having coefficients in $(0,1)$) and $\pm (K_X+D)$ is ample, then the behavior of any Kähler-Einstein metric along $D$ is well understood: it has
so-called cone singularities (cf \cite{Brendle, CGP, Don, JMR}). We now want to show a similar statement in our situation; so let us first recall the notion of cone singularities.

\subsection{Metrics with cone singularities}
Let $X$ be a compact Kähler manifold of dimension $n$, and $\D=\sum a_i \D_i$ an effective $\R$-divisor with simple normal crossing support such that the $a_i$'s satisfy the following inequality: $0 < a_i < 1$. We write $X_0=X\setminus \Supp(\D)$, and we choose non-zero global sections $s_i$ of $\O_X(D_i)$.\\

Our local model is given by the product $X_{\rm mod}=(\mathbb{D}^*) \times \mathbb{D}^{n-r}$ where $\mathbb{D}$ (resp. $\mathbb{D}^*$) is the disc (resp. punctured disc) of radius $1/2$ in $\mathbb C$, the divisor being $D_{\rm mod}=d_1 [z_1=0]+\cdots+d_r [z_r=0]$, with $d_i<1$. We will say that a metric $\om$ on $X_{\rm mod}$ has cone singularities along the divisor $D_{\rm mod}$ if there exists $C>0$ such that 
\[C^{-1} \omega_{\rm mod} \le \om \le C \,\om_{\rm mod}\]
where
\[\om_{\rm mod}:=\sum_{j=1}^r \frac{i dz_j\wedge d\bar z_j}{|z_j|^{2d_j}} +\sum_{j=r+1}^n i dz_j\wedge d\bar z_j \] 
is simply the product metric of the standard cone metric on $(\mathbb{D}^*)^r$ and the euclidian metric on $\mathbb{D}^{n-r}$.\\

This notion makes sense for global (Kähler) metrics $\om$ on the manifold $X_0$; indeed, we can require that on each trivializing chart of $X$ where the pair $(X,\D)$ becomes $(X_{\rm mod},D_{\rm mod})$ (those charts cover $X$), $\om$ is equivalent to $\om_{\rm mod}$ just like above; of course this does not depend on the chosen chart.\\

In our case, we are going to deal with Kähler metrics no more on the whole $X_0$ but on some Zariski open subset, more precisely $X_0\cap \amp$ for some semi-positive and big (or nef and big) class $\alpha$ (in Theorem B, we work on $X_0\sm \Supp(E)$ for example). Therefore one needs to make precise what we will call cone singularities for such a metric. Indeed, $\amp$ being non compact in general (more precisely as soon as $\alpha$ is not ample), the bi-Lipschitz constant comparing the initial and the model cone metric in each local charts may not be chosen uniformly for all charts covering $\amp$. So we just do not ask it to be uniform, and therefore only care on what happens on compact subsets of $X_0 \cap \amp$.\\

\subsection{Statement of the main result}

We have seen in the introduction that Theorem A, which asserts that (under some assumptions) the Kähler-Einstein metric for a klt pair has cone singularities along the boundary (restricted to some suitable open subset). We also explained in the introduction and the previous section \ref{red} how to relate this assertion to some particular properties of solutions of degenerate Monge-Ampère equations. So let us now fix the set-up.\\

Let $X$ be a compact Kähler manifold of dimension $n$, $\alpha$ a nef and big class class and $\theta\in \alpha$ a smooth representative. Let $E=\sum c_j E_j$ be an effective $\R$-divisor with snc support such that $\alpha -E$ is Kähler. Let also $D= \sum a_i D_i$ an effective divisor with snc support such that $E$ and $D$ have no common components, and that $E+D$ has snc support. 

We write $X_0=X\setminus (\Supp(E) \cup \Supp(D))$, we choose non-zero global sections $t_j$ of $\O_X(E_j)$ and $s_i$ of $\O_X(D_i)$, and we choose some real numbers $b_j>-1$. We choose some smooth hermitian metrics on those bundles which we normalize so that $\int_X \prod |t_j|^{2b_j} \prod |s_i|^{-2a_i}dV=\vol(\alpha)$. Finally, if $\vp$ is a $\theta$-psh function, we denote by $\MA(\vp)$ the non-pluripolar product $\la (\theta+\ddc \vp)^n\ra$.\\

\begin{theo}
\phantomsection
We assume that the coefficients of $D$ satisfy the inequalities $a_i \ge 1/2$ for all $i$. 
Then any solution with full Monge-Ampère mass of
\[\MA(\vp)= \prod |t_j|^{2b_j} \frac{e^{\lambda \vp}dV}{\prod |s_i|^{2a_i}}\]
defines a smooth metric on $X_0$ having cone singularities along $D$.
\end{theo}

\section{Proof of Theorem B}
The strategy of the proof is to regularize the Monge-Ampère equation into an equation with smooth right-hand side. The stake will then consist in obtaining uniform estimates on the compact subsets of $X_0$ (cf Proposition \eqref{mp}). In order to simplify the notations, we will use the same notation for $\Supp(E)$ and $E$ whenever no confusion might result from this abuse of notation.

\subsection{Regularization}
\label{subs:approx}
We start from our solution $\vp$, and we regularize the Monge-Ampère equation in two ways: first, we regularize $\vp$ by approximating it with a decreasing sequence of smooth quasi-psh functions $\tau_{\ep}$ satisfying 
\begin{equation}
\label{unif}
\ddc \te \ge -C \om
\end{equation}
for some (Kähler) form $\om$. This is possible thanks to Demailly's regularization theorem \cite{D1, D2}. In particular the $\te$'s are uniformly upper bounded by $\sup \tau_1$ for instance. Then, we consider the following equation (in $\vpe$):

\[\la(\theta+\ddc \vpe)^n\ra = \prod |t_j|^{2b_j} \frac{e^{\lambda \te}dV}{\prod (|s_i|^2+\ep^2)^{a_i}} \leqno{(\MA_{\ep})}\]
By multiplying $dV$ with a constant, we can make sure that the total mass of the RHS is $\vol(\alpha)$); this constant depends on $\ep$, but in a totally harmless way because $ \frac{e^{\lambda \te}\prod |t_j|^{2b_j}}{\prod (|s_i|^2+\ep^2)^{a_i}} $ is uniformly bounded in $L^1(dV)$ : this is clear if $\lambda \ge 0$ and follows from the monotone convergence combined with Theorem \ref{zln} if $\lambda<0$: indeed, this result of \cite{BBEGZ} shows that $e^{-\vp}\in L^p(dV)$ for all $p\ge 1$. Therefore we can assume that the volume form is already normalized.\\

By \cite{BEGZ}, we know that $(\MA_{\ep})$ has a unique solution $\vpe$ which is $\theta$-psh and has minimal singularities. One could also deduce a uniform estimate, but at that point we do not really need it, and we will anyway recover it implicitly with Proposition \ref{mp}.

\subsection{Laplacian estimates}

In this section, we explain in Proposition \ref{mp} how to obtain laplacian estimates for our regularized solutions: this is the key result for the proof of the main theorems. It is an adaptation to \cite[Theorem 10.1]{BBEGZ} in the nef and big case, with a slight refinement (cf point $(iii)$) which will be crucial for us. Before we can state the result, let us introduce some notation. \\

We recall that $\alpha$ is a nef and big class and $E$ is an effective $\R$-divisor such that $\alpha-E$ is Kähler. We choose $s_E$ a non-zero section of the $\Ox(E)$, and choose some smooth hermitian metric $h$ on this $\R$-line bundle; we set \[\rho:= \log |s_E|_h^2\] (actually $s_E$ is not well-defined if $E$ is not a $\Z$-divisor, but $\rho$ is). Then, we define 
\[\omr:= \theta - \Theta_h(E)\]
which is a Kähler form if $h$ is properly chosen ($\Theta_h(E)$ is the Chern curvature of the hermitian $\R$-line bundle $(\Ox(E),h)$). As \[\theta+\ddc \rho = \omr+[E]\] we see that $\omr$ coincide with $\theta+\ddc \rho$ on 
$X\setminus E$. 
 
We emphasize that $\omr$ depends on the smooth hermitian metric chosen on $E$, or equivalently it depends on $\rho$, hence the notation. In the following proposition, we make explicit the precise dependence in $h$ (or $\rho$) in the laplacian estimates obtained in \cite[Theorem 10.1]{BBEGZ}:

\begin{prop}
\phantomsection
\label{mp}
With the previous notations, let $\ppm$ be quasi-psh functions on $X$  such that $e^{-\psm}\in L^p(dV)$ for some $p>1$, and satisfying $\int_X  e^{\pp-\psm} \omr^n = \vol(\alpha)$ . Let $\vp$ be a $\theta$-psh function solution of 
\[\la (\theta + \ddc \vp)^n \ra = e^{\pp-\psm}\omr^n \]
and assume given a constant $C>0$ such that 
\begin{enumerate}
\item[$(i)$] $\ddc \pp \ge -C \omr  \,\,$ and $\, \,\sup_X \pp \le C$; 
\item[$(ii)$]  $\ddc \psm \ge -C \omr  \,\,$ and $\, \, |\!|e^{-\psm} \omr^n/dV |\!|_{L^p(dV)}  \le C$; 
\item[$(iii)$] $\omr \ge C^{-1}\om$ and the holomorphic bisectional curvature of $\om_{\rho}$ is bounded below on $X$ by $-C$.
\end{enumerate}
Then there exist $A,B>0$ depending only on $\theta, p$ and $C$ such that on $X\setminus E$:
\[0 \le \theta + \ddc \vp  \le A e^{-B \rho-\psm} \omr\] 
\end{prop}

\begin{proof}
We will only treat the case where $\ppm$ are smooth; the general case can be reduced to the smooth case by regularization exactly in the same way that in \cite{BBEGZ}, the only difference being the use of the degenerate version of Ko\l odziej's stability theorem for \textit{big} classes given in \cite[Theorem C]{GZ11}.

For $t>0$ (say $t\le 1$), we consider the Kähler form $\omt$ on $X$ defined by $\omt:=(1+t)\omr$; we also define $\theta_t:=\theta+ t\omr$. As $\alpha$ is nef, $\{\theta_t\}$ is a Kähler class for all $t>0$. 
Therefore there exists a unique normalized (smooth) $\theta_t$-psh function $\vpt$ such that 
\[(\theta_t+ \ddc \vpt)^n = e^{\pp-\psm}e^{c_t}\omr^n\]
where $e^{c_t}=\vol(\alpha+t\{\omr\})/\vol(\alpha)$. As $\{\omr\}$ is independent of $\rho$ ($E$ is fixed), $c_t$ is uniformly bounded. 

We want to obtain an estimate $|\Delta_{\omr} \vp_t | \le A e^{-B \rho-\psm}$ on $X\sm E$; as  $\omr\ge \frac 12 \omt$, 
it will be enough to show the same estimate with $\Delta_{\omt} \vpt$. \\

To begin with, thanks to assumptions $(i)$ and $(ii)$, we have a $\mathscr C^0$ estimate given by \cite{BEGZ}, and recalled in section \ref{sct}: $\vpt \ge V_{\theta_t}-M$. Besides, it is easy to see that $V_{\theta_t}$ decreases to $V_{\theta}$ when $t\to 0$, so that $\vpt \ge V_{\theta}-M$ for all $t\ge 0$.\\ 

\textit{For the rest of the proof, we will work on $X\setminus E$, so that $\omr$ actually coincide with $\theta+\ddc \rho$}.\\

At this point of the proof, one cannot use Siu's formula \cite[pp. 98-99]{Siu} as in \cite{BBEGZ} anymore because we do not have an uniform upper bound on the scalar curvature of $\omt$. Instead, we use a variant of Siu's formula given in \cite[Lemma 2.2]{CGP} involving only a bound on the bisectional curvature of $\omt$; it yields: 
\[\Delta_{\omt'}\left( \log \tr_{\omt} (\omt')\right) \ge \frac{\Delta_{\omt}( \pp-\psm)}{ \tr_{\omt} (\omt')} -C\tr_{\omt'} (\omt)\]
where $\omt'= \theta_t+\ddc \vpt$. We notice, as in \cite{BBEGZ}, that even if \cite[Lemma 2.2]{CGP} is stated for two cohomologous forms, the last formula is valid because the computations are local, and locally all forms are cohomologous.\\

Using both inequalities $\omt \ge \omr$ and $n \le  \tr_{\omt} (\omt')  \tr_{\omt'} (\omt)$, we get then a constant $A_0>0$ under control such that: 
\[\Delta_{\omt'}\left( \log \tr_{\omt} (\omt')\right) \ge -\frac{\Delta_{\omt}\psm}{ \tr_{\omt} (\omt')} -A_0\tr_{\omt'} (\omt)\]

Now we use the computations of \cite{BBEGZ} based on \cite[Lemma 3.2]{Paun}: as $C \omt + \ddc \psm \ge 0$, we have 
\[0\le C \omt + \ddc \psm \le \tr_{\omt'}(C \omt + \ddc \psm) \omt' \]
Taking the trace with respect to $\omt$ gives:
\[0\le nC  + \Delta_{\omt} \psm \le \left(C\tr_{\omt'}(\omt) + \Delta_{\omt'} \psm\right) \tr_{\omt}(\omt') \]
so that 
\[ \Delta_{\omt'} \psm \ge -\frac{nC  + \Delta_{\omt} \psm}{\tr_{\omt}(\omt')}- C\tr_{\omt'}(\omt)\]
and therefore:
\begin{equation}
\label{eq1}
\Delta_{\omt'}\left( \log \tr_{\omt} (\omt')+\psm \right) \ge -A_1 \tr_{\omt'}(\omt)
\end{equation}
for some constant $A_1>0$ under control. \\

\noindent
Now, if we set \[u_t:=\vp_t-\rho\]
then $\omt'=\omt +\ddc u_t$ so that $n= \tr_{\omt'}(\omt) +\Delta_{\omt'} u_t$. Equation \eqref{eq1} gives us two positive constants $A_2,A_3$ under control satisfying:
\[ \Delta_{\omt'}( \log \tr_{\omt} (\omt') +\psm-A_2 u_t) \ge  \tr_{\omt'} (\omt)-A_3 \]

We want now to apply as usual the maximum principle to the term inside the laplacian in the right hand side. To ensure we can do this, we must check that the function 
\[H_t:=\log \tr_{\omt} (\omt') +\psm-A_2 u_t\]
 attains its maximum on $X\setminus E$. This is a qualitative problem, and of course we do not ask any kind of uniformity here. 
We know that $\psm$ is bounded (but we do not have uniform bounds), moreover $\tr_{\omt} (\omt') \le C \tr_{\om} (\omt')$ by assumption $(iii)$, and as $\omt'$ is smooth, this last quantity is bounded above. Finally, $-u_t = \rho-\vpt$ is upper bounded  and tends to $-\infty$ near $E$. We conclude that $H_t$
attains its maximum on at some point $x_t\in X\sm E $, so that  $\tr_{\omt'} (\omt)(x_t) \le A_3$ is under control.\\

Using the basic inequality 
\[ \tr_{\omt} (\omt') \le n \left(\omt'^n/\omt^n\right) \left(\tr_{\omt'} (\omt)\right)^{n-1} \] 
and the inequality 
\[\omt'^n/\omt^n = e^{\pp-\psm} \frac{e^{c_t}\omr^n}{\omt^n} \le e^{c_t+\pp-\psm} \]
we get
\[\log \tr_{\omt} (\omt')  \le -\psm +(n-1) \log \tr_{\omt'} (\omt) +A_4\]
with $A_4$ under control, and therefore
\[H_t \le (n-1) \log \tr_{\omt'} (\omt) -A_2 u_t+A_4\]
so that 
\[\sup_{X\sm E} H = H(x_t) \le A_5- A_2u_t(x_t)\]
Therefore, one has, for any $x\in X\setminus E$:
\begin{eqnarray*}
(\log \tr_{\omt} (\omt') +\psm)(x) &=& H(x)+A_2 u_t(x) \\
&\le & H(x_t)+A_2 u_t(x)\\
& \le & A_5+A_2(u_t(x)-u_t(x_t))\\
& \le & A_6 +A_2u_t(x)
\end{eqnarray*}

Indeed, $\vpt \ge V_{\theta}-M$ and $\rho$ is $\theta$-psh, thus $\vp_t \ge \rho - A_7$  so that $\inf_{X\setminus E} \vpt-\rho$ is uniformly bounded from below. As $\vp_t$ is normalized, $u_t\le -\rho$, so that one finally gets $A,B>0$ under control and satisfying: $\log \tr_{\omt} (\omt') +\psm \le A-B \rho$, which is what we were looking for.
\end{proof}

\begin{rema}
\phantomsection
\label{smo}
Combined with Evans-Krylov's theorem, this proposition shows that any Kähler-Einstein metric attached to a klt pair $(X,D)$ (satisfying e.g. $K_X+D$ ample, the other cases being similar) is smooth on $X_{\rm reg} \setminus \Supp(D)$. Indeed, if we work on a suitable log-resolution $X'\overset{\pi}{\longrightarrow} X$ of $(X,D)$ with $K_{X'}+D'= \pi^*(K_X+D)$, then the KE metric (viewed on $X'$) written as usual $\theta'+\ddc \vp'$ satisfies an equation of the form $(\theta'+\ddc \vp')^n = e^{\vp'-\vp_D'}dV$) and we just have to apply the previous proposition with $(\pp,\psm)=( \vp'+\vp_{D_-'},\vp_{D_+'})$ if $D'=D_+'-D_-'$ is the decomposition of $D'$ into its positive and negative part.
\end{rema}

\subsection{Approximation of the cone metric}
\label{app}

We now recall the global approximation of a cone metric, as explained in \cite[Section 3]{CGP} for instance.\\
Let $\om$ be a Kähler form on some compact Kähler manifold $Y$ carrying a $\R$-divisor $F=\sum c_k F_k$ with simple normal crossing support and coefficients $c_k \in ]0,1[$. Then for any sufficiently small $\ep>0$, there exists a smooth function $\psie$ such that the form $\ome$ on $Y$ defined by
\[ \ome:=\om+\ddc \psi_{\ep} \]
satisfies the following properties:
\begin{enumerate}
\item[$\cdot$] $\ome$ dominates a fixed Kähler form on $Y$;
\item[$\cdot$] $\psi_{\ep}$ is uniformly bounded (on $Y$) in $\ep$;
\item[$\cdot$] When $\ep$ goes to $0$, $\ome$ converges to some Kähler metric on $Y\setminus \Supp(F)$ having cone singularities along $F$.\\
\end{enumerate}

Our goal is now to apply this construction for some suitable Kähler form $\om$, and try to apply the laplacian estimates obtained in the previous section in order to get $C^{-1}\ome \le \theta+\ddc \vpe \le C\ome$ on compact subsets of $X\setminus E$.\\

We know that for some hermitian metric $h$ on $E$, the form $\omr=\theta -\Theta_h(E)$ is a Kähler form. So we can apply the previous approximation process to $Y=X$, $F=D$, and $\om=\omr$. We get a sequence of smooth functions $\psie$ such that $\omre:=\omr+\ddc \psie$ is a Kähler form satisfying the three conditions above. Moreover, $\omre$ corresponds to the Kähler metric $\theta-\Theta_{h_{\ep}}(E)$ where $h_{\ep}:=he^{\psie}$, or equivalently $\rhoe:=\rho+\psie $, and we still have that $\omre$ and $\theta+\ddc \rhoe$ coincide on $X\setminus E$.

\subsection{End of the proof}

Recall now that we try to understand the behaviour of the solution of 
\[\la(\theta+\ddc \vpe)^n\ra = \prod |t_j|^{2b_j} \frac{e^{\lambda \te}dV}{\prod (|s_i|^2+\ep^2)^{a_i}} \leqno{(\MA_{\ep})}\]
If we use the metric $\omre$ as new reference, equation $(\MA_{\ep})$ may be rewritten: 
\[\la(\theta+\ddc \vpe)^n\ra=\prod |t_j|^{2b_j} e^{\lambda \te + F_{\ep}} \omre^n\]
where
\[F_{\ep}=\log \left( \frac{dV}{\prod (|s_i|^2+\ep^2)^{a_i} \omre^n}\right). \]

\vspace{5mm}
In order to use the same notations as Proposition \ref{mp}, we set: 
\[ \pp:=\sum_{b_j>0} b_j \log |t_j|^2+\lambda \te+ F_{\ep}, \quad \psm:=\sum_{b_j<0} -b_j\log |t_j|^2\]
if $\lambda \ge 0$, and 
\[ \pp:=\sum_{b_j>0} b_j \log |t_j|^2+ F_{\ep}, \quad \psm:=-\lambda \te +\sum_{b_j<0} -b_j\log |t_j|^2\]
else. We should add that despite the notations, $\pp$ and $\psm$ actually depend on $\ep$.

\noindent
With these notations, equation $(\MA_{\ep})$ becomes: 
\[\la(\theta+\ddc \vpe)^n\ra=e^{\pp-\psm} \omre^n \leqno{(\MA_{\ep}')} \]

\vspace{3mm}
On the compact subsets of $X\setminus E$, $\omre$ converge to a Kähler metric with cone singularities along $D$. Therefore the proof of Theorem B boils down to showing that on each relatively compact open subset $U \Subset X\setminus E$, there exists a constant $C_U>0$ such that on $U$, we have for each $\ep>0$: 
\[ C_U^{-1} \omre \le \theta+\ddc \vpe \le C_U \omre \]
We fix such a open subset $U$. We want to apply Proposition \ref{mp} in our situation, so we need to check that there exists $C>0$ independent of $\ep$ satisfying: 

\begin{enumerate}
\item[$(i)$] $\ddc \pp \ge -C \omre  \,\,$ and $\, \,\sup_X \pp \le C$; 
\item[$(ii)$]  $\ddc \psm \ge -C \omre  \,\,$ and $\, \, |\!|e^{-\psm} \omre^n/dV'|\!|_{L^p}  \le C$; 
\item[$(iii)$] $\omre \ge C^{-1}\om$ and the holomorphic bisectional curvature of $\omre$ is bounded from below (on $X$) by $-C$.\\
\end{enumerate}

Let us begin with $(ii)$. The first thing to check is that $\ddc \te \ge -C \omre$. As $\omre$ dominates some uniform Kähler form (cf construction of the approximate cone metric, section \ref{app}), this follows from 
\eqref{unif}. Then, we have to check the uniform $L^p$ integrability condition.  We may assume that $\lambda \le 0$;  let us fix $\delta>0$ small enough, and more precisely 
\[\delta <  \min_{j} \, \frac{1}{b_j^-}-1\] 
where $b_j^- = \max(-b_j, 0)$.\\
Then by monotone convergence, $e^{\lambda(1+\delta) \te}$ is uniformly in $L^q$ for all $q \ge 1$ because so is $e^{\lambda(1+\delta) \vp}$ by Theorem \ref{zln}; we will fix an appropriate $q$ later. Moreover, by the normal crossing property and the klt condition, $(\prod_{b_j<0} |t_j|^{2b_j})^{1+\delta}$ is also uniformly in $L^{1+\eta}(dV)$ for some $\eta>0$ small enough. Therefore, taking $q=1+\frac{1}{\eta}$, we get the result. \\

There remain two non-trivial estimates to check, namely that $\ddc F_{\ep} \ge -C \omre$, and that the holomorphic bisectional curvature of $\omre$ is uniformly bounded below on $X\setminus E$. This is precisely at this point of the proof that we use in a crucial manner the assumption that the coefficients of $D$ are in $[\frac 12, 1)$. 

Indeed the first estimate is already obtained in \cite[\textsection 4.5]{CGP} in the slightly weaker form $\Delta_{\omre} F_{\ep} \ge -C$, and is proven in the desired form in \cite[\textsection 4.2.3]{G12}. As for the second estimate, concerning the curvature of $\omre$, it is also proven in \cite[\textsection 4.3-4.4]{CGP}. Finally, as for the bound $F_{\ep} \le C$, it is quite easy and explained in the same references. \\

Therefore, we can legitimately apply Proposition \ref{mp} to equation $(\MA_{\ep}')$, which ends the proof of Theorem B.

\begin{rema}
Let us emphasize that we may use the computations of \cite{CGP} to estimate e.g. the curvature of $\omre$ because this last metric is of the form $\omr+\ddc \psie$ for some \textit{fixed} Kähler metric $\omr$ on the whole $X$. It is not clear to us whether the same regularization argument could be performed directly on $\amp$ (ie without choosing a suitable compactification) as in \cite[Theorem 10.1]{BBEGZ}.
\end{rema}

\subsection{A slight generalization}

We should mention that we do not use in the proof of Theorem B the fact that the factors $|t_j|^{-2b_j}$ in the RHS of the Monge-Ampère equation are coming from the "non-ample part" $E$. Therefore, on could also choose $t_j$ to be a section of $\Ox(D_j)$ for some component $D_j$ with coefficient $a_j<1/2$. This leads to the following generalization of Theorem A:

\begin{theo}
\phantomsection
\label{gen}
Let $(X,D)$ be a klt pair. We write $D=D_{\ge \frac 1 2}+D_{<\frac 1 2}=\sum_{a_i\ge \frac 1 2} D_i+\sum_{a_i<\frac 1 2}D_i$, and we set $\mathrm{LS}(X,D_{\ge \frac 1 2}):= \{x\in X; (X,D_{\ge \frac 1 2})$ is log-smooth at $x\}$.
\begin{enumerate}
\item[$(i)$] If $K_X+D$ is big, then the Kähler-Einstein metric of $(X,D)$ has cone singularities along $D_{\ge \frac 1 2}$ on $\mathrm{LS}(X,D_{\ge \frac 1 2})\cap \mathrm{Amp}(K_X+D)\sm \Supp(D_{<\frac 1 2})$. 
\item[$(ii)$] If $-(K_X+D)$ is ample, then any Kähler-Einstein metric of $(X,D)$ has cone singularities along $D_{\ge \frac 1 2}$ on $\mathrm{LS}(X,D_{\ge \frac 1 2})\sm \Supp(D_{<\frac 1 2})$.
\end{enumerate}
\end{theo}

\bibliographystyle{smfalpha}
\bibliography{biblio.bib}

\end{document}